\documentclass[10pt,reqno]{amsart}

\usepackage{amssymb}
\usepackage{amsmath}
\usepackage{amsfonts}
\usepackage{amscd}
\usepackage{mathrsfs}
\usepackage{graphicx}
\usepackage{yhmath}
\usepackage{mathdots}

\numberwithin{equation}{section}
\newtheorem{theorem}{Theorem}[section]
\newtheorem{claim}[theorem]{Claim}

\newtheorem{proposition}[theorem]{Proposition}

\newcommand{\n}[1]{\left\vert#1\right\vert}

\begin{document}
\title[]{Algebraic differential independence regarding \\ the Riemann $\boldsymbol{\zeta}$-function and the Euler $\boldsymbol{\Gamma}$-function}

\author[]{Qi Han and Jingbo Liu$^\star$}

\address{Department of Mathematics, Texas A\&M University, San Antonio, Texas 78224, USA
\vskip 2pt Email: {\sf qhan@tamusa.edu (QH) \hspace{0.6mm} jliu@tamusa.edu (JL)}}



\thanks{$^\star$Jingbo Liu is the corresponding author.}

\begin{abstract}
  In this paper, we prove that $\boldsymbol{\zeta}$ cannot be a solution to any nontrivial algebraic differential equation whose coefficients are polynomials in $\boldsymbol{\Gamma},\boldsymbol{\Gamma}^{(n)}$ and $\boldsymbol{\Gamma}^{(\ell n)}$ over the ring of polynomials in $\mathbf{C}$, where $\ell,n\geq1$ are positive integers.
\end{abstract}

\maketitle

\section{Introduction and main results}
It is a celebrated result of H\"{o}lder \cite{Ho} in 1887 that the Euler gamma-function
\begin{equation}
\boldsymbol{\Gamma}(z)=\int_0^{+\infty}t^{z-1}e^{-t}dt\nonumber
\end{equation}
cannot satisfy any nontrivial algebraic differential equation whose coefficients are polynomials in $\mathbf{C}$.
That is, if $P(v_0,v_1,\ldots,v_n)$ is a polynomial of $n+1$ variables with polynomial coefficients in $z\in\mathbf{C}$ such that $P(\boldsymbol{\Gamma},\boldsymbol{\Gamma}',\ldots,\boldsymbol{\Gamma}^{(n)})(z)\equiv0$ for $z\in\mathbf{C}$, then necessarily $P\equiv0$.
Hilbert \cite{Hi}, in his lecture addressed at the International Congress of Mathematicians at Paris in 1900 for his famous 23 problems, stated in Problem 18 that the Riemann zeta-function
\begin{equation}
\boldsymbol{\zeta}(z)=\sum_{n=1}^{+\infty}\frac{1}{n^z}\nonumber
\end{equation}
cannot satisfy any nontrivial algebraic differential equation whose coefficients are polynomials in $\mathbf{C}$, and this problem was solved with great generality (towards a question posed by Hilbert in his Problem 18) by Mordukhai-Boltovskoi \cite{MB} and Ostrowski \cite{Os}, independently.

It is well-known that $\boldsymbol{\zeta}$ and $\boldsymbol{\Gamma}$ are related by the Riemann functional equation
\begin{equation}\label{Eq1.1}
\boldsymbol{\zeta}(1-z)=2^{1-z}\pi^{-z}\cos\Big(\frac{1}{2}\pi z\Big)\boldsymbol{\Gamma}(z)\boldsymbol{\zeta}(z).
\end{equation}
By virtue of \eqref{Eq1.1}, one knows from Bank and Kaufman \cite{BH} that neither $\boldsymbol{\zeta}$ nor $\boldsymbol{\Gamma}$ can satisfy any nontrivial algebraic differential equation whose coefficients are meromorphic functions, loosely speaking, growing strictly slower than $e^z$ as $\n{z}\to\infty$\footnote{The precise description involves the classical Nevanlinna theory, which is not used in this paper.} or having period $1$; detailed descriptions of this matter and some related results of an entirely different flavor can be found in Li and Ye \cite{LY3}.
In addition, it is interesting to notice Van Gorder \cite{VG} showed that $\boldsymbol{\zeta}$ does satisfy formally an infinite order linear differential equation having analytic coefficients.

Recently, Markus \cite{Ma} proved that $\boldsymbol{\zeta}(\sin(2\pi z))$ cannot satisfy any nontrivial algebraic differential equation whose coefficients are polynomials in $\boldsymbol{\Gamma}$ and its derivatives, and he conjectured that $\boldsymbol{\zeta}$ itself cannot satisfy any nontrivial algebraic differential equation whose coefficients are polynomials in $\boldsymbol{\Gamma}$ and its derivatives, either.
Thus, we are interested in knowing whether there is a nontrivial polynomial $P(u_0,u_1,\ldots,u_m;v_0,v_1,\ldots,v_n)$ such that, for $z\in\mathbf{C}$,
\begin{equation}
P(\boldsymbol{\zeta},\boldsymbol{\zeta}',\ldots,\boldsymbol{\zeta}^{(m)};\boldsymbol{\Gamma},\boldsymbol{\Gamma}',\ldots,\boldsymbol{\Gamma}^{(n)})(z)\equiv0.\nonumber
\end{equation}

Notice the preceding result \cite{BH} cannot be applied to algebraic differential equations involving both $\boldsymbol{\zeta}$ and $\boldsymbol{\Gamma}$ simultaneously, since $\boldsymbol{\zeta}$ and $\boldsymbol{\Gamma}$ grow strictly faster than $e^z$ when $\n{z}\to\infty$; see Ye \cite{Ye}.
In this aspect, some related results on the value distribution properties of $\boldsymbol{\zeta}$ and $\boldsymbol{\Gamma}$ as well as their general extensions can be found, for instance, in Li \cite{Li} and Han \cite{Ha}.

In this paper, we shall prove the main result as follows.

\begin{theorem}\label{thm1.1}
Let $\ell,m,n\geq0$ be nonnegative integers.
Assume $P(u_0,u_1,\ldots,u_m;v_0,v_1,v_2)$ is a polynomial of $m+4$ variables with polynomial coefficients in $z\in\mathbf{C}$ such that
\begin{equation}\label{Eq1.2}
P(\boldsymbol{\zeta},\boldsymbol{\zeta}',\ldots,\boldsymbol{\zeta}^{(m)};\boldsymbol{\Gamma},\boldsymbol{\Gamma}^{(n)},\boldsymbol{\Gamma}^{(\ell n)})(z)\equiv0
\end{equation}
for $z\in\mathbf{C}$.
Then, necessarily the polynomial $P$ must be identically equal to zero.
\end{theorem}

To prove our result, we need a renowned theorem from Voronin \cite{Vo}.

\begin{proposition}\label{prop1.2}
Fix $x\in\big(\frac{1}{2},1\big)$ for $z=x+\mathtt{i}y\in\mathbf{C}$.
Define
\begin{equation}
\gamma(y):=(\boldsymbol{\zeta}(x+\mathtt{i}y),\boldsymbol{\zeta}'(x+\mathtt{i}y),\ldots,\boldsymbol{\zeta}^{(m)}(x+\mathtt{i}y))\nonumber
\end{equation}
to be a curve in $y$.
Then, $\gamma(\mathbf{R})$ is everywhere dense in $\mathbf{C}^{m+1}$.
\end{proposition}

Note our result is closely related to Li and Ye \cite{LY2}, while our main scheme follows that of it; yet, we consider higher order derivatives of $\boldsymbol{\Gamma}$ rather than just $\boldsymbol{\Gamma}'$ and $\boldsymbol{\Gamma}''$.
On the other hand, some remarks seem necessary regarding the work \cite{LY1,LY3}.
In \cite{LY1}, a quite restrictive condition on the polynomials involved there was added in its main theorem, and our result here recovers its corollary as a special case for $\ell=n=1$; in its corollary, the restrictive condition was dropped, and general polynomials were used but only for $\boldsymbol{\Gamma}$ and $\boldsymbol{\Gamma}'$.
As of \cite{LY3}, using its terminology, only $\boldsymbol{\zeta}$ and the family of {\sl distinguished polynomials} of $\boldsymbol{\Gamma},\boldsymbol{\Gamma}',\ldots,\boldsymbol{\Gamma}^{(n)}$ were considered.

\section{Proof of Theorem \ref{thm1.1}}
Let $P(u_0,u_1,\ldots,u_m;v_0,v_1,v_2)$ be a polynomial in its arguments whose coefficients are polynomials in $z\in\mathbf{C}$ such that \eqref{Eq1.2} is satisfied.
In view of the discussions in \cite[Section 2]{Ma}, one can simply assume that the coefficients of $P$ are constants.
Denote by
\begin{equation}
\Lambda:=\left\{\lambda:=(\lambda_0,\lambda_1,\lambda_2):\lambda_0,\lambda_1,\lambda_2~\text{are nonnegative integers}\right\}\nonumber
\end{equation}
a triple-index set having a finite cardinality, and define
\begin{equation}
\Lambda_p:=\left\{\lambda\in\Lambda:\n{\lambda}=p~\text{with}~\n{\lambda}:=\lambda_0+\lambda_1+\lambda_2\right\}\nonumber
\end{equation}
and
\begin{equation}
\Lambda^\star_q:=\left\{\lambda\in\Lambda:\n{\lambda}^\star=q~\text{with}~\n{\lambda}^\star:=\lambda_1+\ell\lambda_2\right\}.\nonumber
\end{equation}
Then, there exists a nonnegative integer $L$ such that
\begin{equation}
P(u_0,u_1,\ldots,u_m;v_0,v_1,v_2)=\sum_{p=0}^L\sum_{\lambda\in\Lambda_p}a_\lambda(u_0,u_1,\ldots,u_m)v_0^{\lambda_0}v_1^{\lambda_1}v_2^{\lambda_2},\nonumber
\end{equation}
where $a_\lambda(u_0,u_1,\ldots,u_m)$ is a polynomial of $m+1$ variables with constant coefficients.
Set, for each $p=0,1,\ldots,L$, the associated homogeneous polynomial to be
\begin{equation}
P_p(u_0,u_1,\ldots,u_m;v_0,v_1,v_2):=\sum_{\lambda\in\Lambda_p}a_\lambda(u_0,u_1,\ldots,u_m)v_0^{\lambda_0}v_1^{\lambda_1}v_2^{\lambda_2}.\nonumber
\end{equation}

Rearrange $P_p(u_0,u_1,\ldots,u_m;v_0,v_1,v_2)$ if necessary in $v:=(v_1,v_2,v_3)$ in the ascending order of $q=\n{\lambda}^\star$, along with the standard lexicographical order when two or more terms having the same indices $p,q$ appear, to find a nonnegative integer $M_p$ such that
\begin{equation}\label{Eq2.1}
P_p(u_0,u_1,\ldots,u_m;v_0,v_1,v_2)=\sum_{q=0}^{M_p}\sum_{\lambda\in\Lambda_p\cap\Lambda^\star_q}a_\lambda(u_0,u_1,\ldots,u_m)v_0^{\lambda_0}v_1^{\lambda_1}v_2^{\lambda_2}.
\end{equation}
As a consequence, it follows that
\begin{equation}\label{Eq2.2}
P(u_0,u_1,\ldots,u_m;v_0,v_1,v_2)=\sum_{p=0}^L\sum_{q=0}^{M_p}\sum_{\lambda\in\Lambda_p\cap\Lambda^\star_q}a_\lambda(u_0,u_1,\ldots,u_m)v_0^{\lambda_0}v_1^{\lambda_1}v_2^{\lambda_2}.
\end{equation}

\begin{claim}\label{clm2.1}
Assume \eqref{Eq1.2} holds.
Then, for each $0\leq p\leq L$ and all $z\in\mathbf{C}$, one has
\begin{equation}
P_p(\boldsymbol{\zeta},\boldsymbol{\zeta}',\ldots,\boldsymbol{\zeta}^{(m)};\boldsymbol{\Gamma},\boldsymbol{\Gamma}^{(n)},\boldsymbol{\Gamma}^{(\ell n)})(z)\equiv0.\nonumber
\end{equation}
\end{claim}

\begin{proof}
Suppose $p_0$ is the smallest index among $\left\{0,1,\ldots,L\right\}$ such that
\begin{equation}
P_{p_0}(\boldsymbol{\zeta},\boldsymbol{\zeta}',\ldots,\boldsymbol{\zeta}^{(m)};\boldsymbol{\Gamma},\boldsymbol{\Gamma}^{(n)},\boldsymbol{\Gamma}^{(\ell n)})(z)\not\equiv0.\nonumber
\end{equation}
Then, we have
\begin{equation}\label{Eq2.3}
\begin{aligned}
&P_{p_0}\Big(\boldsymbol{\zeta},\boldsymbol{\zeta}',\ldots,\boldsymbol{\zeta}^{(m)};1,\frac{\boldsymbol{\Gamma}^{(n)}}{\boldsymbol{\Gamma}},\frac{\boldsymbol{\Gamma}^{(\ell n)}}
{\boldsymbol{\Gamma}}\Big)(z)\\
=\,&\frac{P_{p_0}(\boldsymbol{\zeta},\boldsymbol{\zeta}',\ldots,\boldsymbol{\zeta}^{(m)};\boldsymbol{\Gamma},\boldsymbol{\Gamma}^{(n)},\boldsymbol{\Gamma}^{(\ell n)})(z)}
{\boldsymbol{\Gamma}^{p_0}(z)}\not\equiv0.
\end{aligned}
\end{equation}

Define the digamma function $f:=\frac{\boldsymbol{\Gamma}'}{\boldsymbol{\Gamma}}$, and introduce inductively
\begin{equation}
{\small\begin{aligned}
\frac{\boldsymbol{\Gamma}'}{\boldsymbol{\Gamma}}&=f\\
&=f\Big[1+\frac{f'}{f^2}(c_1+\varepsilon_1)\Big]~\text{for}~c_1=0~\text{and}~\varepsilon_1=0,\\
\frac{\boldsymbol{\Gamma}''}{\boldsymbol{\Gamma}}&=\Big(\frac{\boldsymbol{\Gamma}'}{\boldsymbol{\Gamma}}\Big)'+\Big(\frac{\boldsymbol{\Gamma}'}{\boldsymbol{\Gamma}}\Big)^2=f'+f^2\\
&=f^2\Big[1+\frac{f'}{f^2}(c_2+\varepsilon_2)\Big]~\text{for}~c_2=1~\text{and}~\varepsilon_2=0,\\
\frac{\boldsymbol{\Gamma}'''}{\boldsymbol{\Gamma}}&=\Big(\frac{\boldsymbol{\Gamma}''}{\boldsymbol{\Gamma}}\Big)'
+\frac{\boldsymbol{\Gamma}''}{\boldsymbol{\Gamma}}\cdot\frac{\boldsymbol{\Gamma}'}{\boldsymbol{\Gamma}}=f''+3ff'+f^3\\
&=f^3\Big[1+\frac{f'}{f^2}(c_3+\varepsilon_3)\Big]~\text{for}~c_3=3~\text{and}~\varepsilon_3=\frac{f''}{ff'},\\
\frac{\boldsymbol{\Gamma}^{(4)}}{\boldsymbol{\Gamma}}&=\Big(\frac{\boldsymbol{\Gamma}'''}{\boldsymbol{\Gamma}}\Big)'
+\frac{\boldsymbol{\Gamma}'''}{\boldsymbol{\Gamma}}\cdot\frac{\boldsymbol{\Gamma}'}{\boldsymbol{\Gamma}}=f'''+4ff''+3(f')^2+6f^2f'+f^4\\
&=f^4\Big[1+\frac{f'}{f^2}(c_4+\varepsilon_4)\Big]~\text{for}~c_4=6~\text{and}~\varepsilon_4=\frac{f'''}{f^2f'}+4\frac{f''}{ff'}+3\frac{f'}{f^2},\\
\frac{\boldsymbol{\Gamma}^{(5)}}{\boldsymbol{\Gamma}}&=\Big(\frac{\boldsymbol{\Gamma}^{(4)}}{\boldsymbol{\Gamma}}\Big)'
+\frac{\boldsymbol{\Gamma}^{(4)}}{\boldsymbol{\Gamma}}\cdot\frac{\boldsymbol{\Gamma}'}{\boldsymbol{\Gamma}}=f^{(4)}+5ff'''+10f'f''+10f^2f''+15f(f')^2+10f^3f'+f^5\\
&=f^5\Big[1+\frac{f'}{f^2}(c_5+\varepsilon_5)\Big]~\text{for}~c_5=10~\text{and}~
\varepsilon_5=\frac{f^{(4)}}{f^3f'}+5\frac{f'''}{f^2f'}+10\frac{f''}{ff'}+10\frac{f''}{f^3}+15\frac{f'}{f^2},\\
&\vdots\\
\frac{\boldsymbol{\Gamma}^{(n)}}{\boldsymbol{\Gamma}}&=f^n\Big[1+\frac{f'}{f^2}(c_n+\varepsilon_n)\Big]~\text{upon assumption, and then we can deduce that}\\
\frac{\boldsymbol{\Gamma}^{(n+1)}}{\boldsymbol{\Gamma}}&=\Big(\frac{\boldsymbol{\Gamma}^{(n)}}{\boldsymbol{\Gamma}}\Big)'
+\frac{\boldsymbol{\Gamma}^{(n)}}{\boldsymbol{\Gamma}}\cdot\frac{\boldsymbol{\Gamma}'}{\boldsymbol{\Gamma}}\\
&=f^{n+1}+f^{n-1}f'(c_n+n+\varepsilon_n)+f^{n-2}f'\varepsilon_n'+[(n-2)f^{n-3}(f')^2+f^{n-2}f''](c_n+\varepsilon_n)\\
&=f^{n+1}\Big[1+\frac{f'}{f^2}(c_{n+1}+\varepsilon_{n+1})\Big]~\text{for}~c_{n+1}=c_n+n=\frac{n(n+1)}{2}~\text{and}\\ &\hspace{49.76mm}\varepsilon_{n+1}=\varepsilon_n+\frac{\varepsilon_n'}{f}+\Big[(n-2)\frac{f'}{f^2}+\frac{f''}{ff'}\Big](c_n+\varepsilon_n).\nonumber
\end{aligned}}
\end{equation}

It is noteworthy that when $n=0$, one has $P(u_0,u_1,\ldots,u_m;v_0)\equiv0$, or when $\ell=1$, one has $P(u_0,u_1,\ldots,u_m;v_0,v_1)\equiv0$, provided for $z\in\mathbf{C}$, we have
\begin{equation}
P(\boldsymbol{\zeta},\boldsymbol{\zeta}',\ldots,\boldsymbol{\zeta}^{(m)};\boldsymbol{\Gamma})(z)\equiv0~\text{or}~
P(\boldsymbol{\zeta},\boldsymbol{\zeta}',\ldots,\boldsymbol{\zeta}^{(m)};\boldsymbol{\Gamma},\boldsymbol{\Gamma}^{(n)})(z)\equiv0,\nonumber
\end{equation}
which follows readily from the analysis in \cite[Corollary]{LY1}.
So, in our subsequent discussions, we without loss of generality always assume $\ell\geq2$ and $n\geq1$.

Define
\begin{equation}\label{Eq2.4}
F:=\frac{\boldsymbol{\Gamma}^{(n)}}{\boldsymbol{\Gamma}}=f^n\Big[1+\frac{f'}{f^2}(c_n+\varepsilon_n)\Big]
\end{equation}
and then rewrite
\begin{equation}\label{Eq2.5}
\frac{\boldsymbol{\Gamma}^{(\ell n)}}{\boldsymbol{\Gamma}}=f^{\ell n}\Big[1+\frac{f'}{f^2}(c_{\ell n}+\varepsilon_{\ell n})\Big]=F^\ell\Big(1+\frac{f'}{f^2}G\Big).
\end{equation}
Here, we set
\begin{equation}\label{Eq2.6}
G:=\Big[\frac{\boldsymbol{\Gamma}^{(\ell n)}}{\boldsymbol{\Gamma}}\frac{\boldsymbol{\Gamma}^\ell}{(\boldsymbol{\Gamma}^{(n)})^\ell}-1\Big]
\frac{(\boldsymbol{\Gamma}')^2}{\boldsymbol{\Gamma}\boldsymbol{\Gamma}''-(\boldsymbol{\Gamma}')^2},
\end{equation}
a meromorphic function in $\mathbf{C}$ satisfying
\begin{equation}
\begin{aligned}
&1+\frac{f'}{f^2}G\\
=\,&\Big[1+\frac{f'}{f^2}(c_{\ell n}+\varepsilon_{\ell n})\Big]\bigg/\Big[1+\frac{f'}{f^2}(c_n+\varepsilon_n)\Big]^\ell\\
=\,&1+[(c_{\ell n}+\varepsilon_{\ell n})-\ell(c_n+\varepsilon_n)]\frac{f'}{f^2}\\
&+\sum_{j=2}^{+\infty}\left\{(c_n+\varepsilon_n)^{j-1}\Big[(c_{\ell n}+\varepsilon_{\ell n})
\binom{-\ell}{j-1}+(c_n+\varepsilon_n)\binom{-\ell}{j}\Big]\right\}\Big(\frac{f'}{f^2}\Big)^j.\nonumber
\end{aligned}
\end{equation}
So, $G$ has the power series representation in terms of $(c_n+\varepsilon_n)\frac{f'}{f^2}$ as follows:
\begin{equation}\label{Eq2.7}
G=\sum_{j=0}^{+\infty}\left\{\Big[(c_{\ell n}+\varepsilon_{\ell n})
\binom{-\ell}{j}+(c_n+\varepsilon_n)\binom{-\ell}{j+1}\Big]\right\}\Big[(c_n+\varepsilon_n)\frac{f'}{f^2}\Big]^j.
\end{equation}

Now, using \eqref{Eq2.1}, \eqref{Eq2.3}, \eqref{Eq2.4} and \eqref{Eq2.5}, for $z\in\mathbf{C}$, we get
\begin{equation}\label{Eq2.8}
\begin{aligned}
&P_{p_0}\Big(\boldsymbol{\zeta},\boldsymbol{\zeta}',\ldots,\boldsymbol{\zeta}^{(m)};1,\frac{\boldsymbol{\Gamma}^{(n)}}{\boldsymbol{\Gamma}},\frac{\boldsymbol{\Gamma}^{(\ell n)}} {\boldsymbol{\Gamma}}\Big)(z)\\
=&\sum_{q=0}^{M_{p_0}}\sum_{\lambda\in\Lambda_{p_0}\cap\Lambda^\star_q}a_\lambda(\boldsymbol{\zeta},\boldsymbol{\zeta}',\ldots,\boldsymbol{\zeta}^{(m)})(z)
\Big(\frac{\boldsymbol{\Gamma}^{(n)}}{\boldsymbol{\Gamma}}\Big)^{\lambda_1}(z)\Big(\frac{\boldsymbol{\Gamma}^{(\ell n)}}{\boldsymbol{\Gamma}}\Big)^{\lambda_2}(z)\\
=&\sum_{q=0}^{M_{p_0}}\sum_{\lambda\in\Lambda_{p_0}\cap\Lambda^\star_q}a_\lambda(\boldsymbol{\zeta},\boldsymbol{\zeta}',\ldots,\boldsymbol{\zeta}^{(m)})(z)
F^{\lambda_1}(z)F^{\ell\lambda_2}(z)\Big(1+\frac{f'}{f^2}G\Big)^{\lambda_2}(z)\\
=&\sum_{q=0}^{M_{p_0}}F^q(z)\sum_{\lambda\in\Lambda_{p_0}\cap\Lambda^\star_q}a_\lambda(\boldsymbol{\zeta},\boldsymbol{\zeta}',\ldots,\boldsymbol{\zeta}^{(m)})(z)
\Big(1+\frac{f'}{f^2}G\Big)^{\lambda_2}(z).
\end{aligned}
\end{equation}
Here, for some $q$, a term with $\n{\lambda}=p_0$ and $\n{\lambda}^\star=q$ may not appear in \eqref{Eq2.8}; if so, one simply deems the coefficient $a_\lambda(u_0,u_1,\ldots,u_m)$ affiliated with this term as zero.

Recall $\lambda_0,\lambda_1,\lambda_2\geq0$ are nonnegative integers.
For each $\lambda=(\lambda_0,\lambda_1,\lambda_2)\in\Lambda_{p_0}\cap\Lambda^\star_q$, one has $\lambda_0=p_0-q+(\ell-1)\lambda_2$ and $\lambda_1=q-\ell\lambda_2$; thus, for fixed $p_0,q$, $\lambda$ is uniquely determined by $\lambda_2$, and vice versa.
Denote the largest $\lambda_2$ by $N_{p_0}$.
Then, \eqref{Eq2.8} can be rewritten as
\begin{equation}\label{Eq2.9}
\begin{aligned}
&P_{p_0}\Big(\boldsymbol{\zeta},\boldsymbol{\zeta}',\ldots,\boldsymbol{\zeta}^{(m)};1,\frac{\boldsymbol{\Gamma}^{(n)}}{\boldsymbol{\Gamma}},\frac{\boldsymbol{\Gamma}^{(\ell n)}} {\boldsymbol{\Gamma}}\Big)(z)\\
=&\sum_{q=0}^{M_{p_0}}F^q(z)\sum_{r=0}^{N_{p_0}}a_{q,r}(\boldsymbol{\zeta},\boldsymbol{\zeta}',\ldots,\boldsymbol{\zeta}^{(m)})(z)\Big(1+\frac{f'}{f^2}G\Big)^r(z).
\end{aligned}
\end{equation}
Here, we set $a_{q,r}(u_0,u_1,\ldots,u_m):=a_\lambda(u_0,u_1,\ldots,u_m)$ for $\lambda=(p_0-q+(\ell-1)r,q-\ell r,r)\in\Lambda$ when applicable; otherwise, we simply set $a_{q,r}(u_0,u_1,\ldots,u_m)$ to be zero.

Define $H:=\frac{f'}{f^2}G$ and expand $(1+H)^r$ to observe that
\begin{equation}
\sum_{r=0}^{N_{p_0}}a_{q,r}(\boldsymbol{\zeta},\boldsymbol{\zeta}',\ldots,\boldsymbol{\zeta}^{(m)})(z)(1+H)^r(z)
=\sum_{r=0}^{N_{p_0}}b_{q,r}(\boldsymbol{\zeta},\boldsymbol{\zeta}',\ldots,\boldsymbol{\zeta}^{(m)})(z)H^r(z).\nonumber
\end{equation}
Here, for fixed $p_0,q$, the polynomials $a_{q,r}(u),b_{q,r}(u)$ satisfy the relations as follows:
\begin{equation}\label{Eq2.10}
\begin{aligned}
b_{q,N_{p_0}}(u)&=a_{q,N_{p_0}}(u),\\
b_{q,N_{p_0}-1}(u)&=a_{q,N_{p_0}-1}(u)+\binom{N_{p_0}}{N_{p_0}-1}a_{q,N_{p_0}}(u),\\
b_{q,N_{p_0}-2}(u)&=a_{q,N_{p_0}-2}(u)+\binom{N_{p_0}-1}{N_{p_0}-2}a_{q,N_{p_0}-1}(u)+\binom{N_{p_0}}{N_{p_0}-2}a_{q,N_{p_0}}(u),\\
&\vdots\\
b_{q,0}(u)&=a_{q,0}(u)+a_{q,1}(u)+\cdots+a_{q,N_{p_0}-1}(u)+a_{q,N_{p_0}}(u),
\end{aligned}
\end{equation}
from which one sees $\left\{a_{q,r}(u)\right\}$ and $\left\{b_{q,r}(u)\right\}$ are mutually uniquely representable of each other, with $a_{q,r}(u):=a_{q,r}(u_0,u_1,\ldots,u_m)$ and $b_{q,r}(u):=b_{q,r}(u_0,u_1,\ldots,u_m)$ for brevity.

Summarizing all the preceding discussions leads to
\begin{equation}\label{Eq2.11}
\begin{aligned}
&P_{p_0}\Big(\boldsymbol{\zeta},\boldsymbol{\zeta}',\ldots,\boldsymbol{\zeta}^{(m)};1,\frac{\boldsymbol{\Gamma}^{(n)}}{\boldsymbol{\Gamma}},\frac{\boldsymbol{\Gamma}^{(\ell n)}} {\boldsymbol{\Gamma}}\Big)(z)\\
=\,&F^{M_{p_0}}(z)\Big[b_{M_{p_0},0}(\vec{\boldsymbol{\zeta}})+b_{M_{p_0},1}(\vec{\boldsymbol{\zeta}})H+\cdots+b_{M_{p_0},N_{p_0}}(\vec{\boldsymbol{\zeta}})H^{N_{p_0}}\Big](z)\\
&+F^{M_{p_0}-1}(z)\Big[b_{M_{p_0}-1,0}(\vec{\boldsymbol{\zeta}})+b_{M_{p_0}-1,1}(\vec{\boldsymbol{\zeta}})H+\cdots+b_{M_{p_0}-1,N_{p_0}}(\vec{\boldsymbol{\zeta}})H^{N_{p_0}}\Big](z)\\
&+\cdots+F(z)\Big[b_{1,0}(\vec{\boldsymbol{\zeta}})+b_{1,1}(\vec{\boldsymbol{\zeta}})H+\cdots+b_{1,N_{p_0}}(\vec{\boldsymbol{\zeta}})H^{N_{p_0}}\Big](z)\\
&+\Big[b_{0,0}(\vec{\boldsymbol{\zeta}})+b_{0,1}(\vec{\boldsymbol{\zeta}})H+\cdots+b_{0,N_{p_0}}(\vec{\boldsymbol{\zeta}})H^{N_{p_0}}\Big](z),
\end{aligned}
\end{equation}
with $(\vec{\boldsymbol{\zeta}})$ being the abbreviation for the vector function $(\boldsymbol{\zeta},\boldsymbol{\zeta}',\ldots,\boldsymbol{\zeta}^{(m)})$.

Recall \eqref{Eq2.3}.
Suppose $b_{q_0,r_0}(u_0,u_1,\ldots,u_m)$ is the first nonzero term in the ordered sequence of polynomials over $\mathbf{C}^{m+1}$ as follows:
\begin{equation}
\begin{aligned}
b_{M_{p_0},0}(u),b_{M_{p_0}-1,0}(u),&\ldots,b_{1,0}(u),b_{0,0}(u),\\
b_{M_{p_0},1}(u),b_{M_{p_0}-1,1}(u),&\ldots,b_{1,1}(u),b_{0,1}(u),\\
&\adots\\
b_{M_{p_0},N_{p_0}-1}(u),b_{M_{p_0}-1,N_{p_0}-1}(u),&\ldots,b_{1,N_{p_0}-1}(u),b_{0,N_{p_0}-1}(u),\\
b_{M_{p_0},N_{p_0}}(u),b_{M_{p_0}-1,N_{p_0}}(u),&\ldots,b_{1,N_{p_0}}(u),b_{0,N_{p_0}}(u).\nonumber
\end{aligned}
\end{equation}

In view of the finiteness of indices, we certainly can find a constant $C_0>1$ and a (sufficiently small) subset $\mathbf{\Omega}$ of $\mathbf{C}^{m+1}$ such that, after appropriate rescaling if necessary,
\begin{equation}
\n{b_{q_0,r_0}(u_0,u_1,\ldots,u_m)}\geq1~\text{and}~\n{b_{q,r}(u_0,u_1,\ldots,u_m)}\leq C_0\nonumber
\end{equation}
uniformly for all $u:=(u_0,u_1,\ldots,u_m)\in\mathbf{\Omega}\varsubsetneq\mathbf{C}^{m+1}$ and for each $0\leq q\leq M_{p_0}$ and $0\leq r\leq N_{p_0}$.
Then, using Proposition \ref{prop1.2}, there exists a sequence of real numbers $\left\{y_k\right\}_{k=1}^{+\infty}$ with $\n{y_k}\to+\infty$ such that $\gamma(y_k)\in\mathbf{\Omega}\varsubsetneq\mathbf{C}^{m+1}$ when $x=\frac{3}{4}$.
So, for $z_k:=\frac{3}{4}+\mathtt{i}y_k\in\mathbf{C}$, one has
\begin{equation}\label{Eq2.12}
\n{b_{q_0,r_0}(\boldsymbol{\zeta},\boldsymbol{\zeta}',\ldots,\boldsymbol{\zeta}^{(m)})(z_k)}\geq1~\text{and}~
\n{b_{q,r}(\boldsymbol{\zeta},\boldsymbol{\zeta}',\ldots,\boldsymbol{\zeta}^{(m)})(z_k)}\leq C_0
\end{equation}
uniformly for all indices $q=0,1,\ldots,M_{p_0}$ and $r=0,1,\ldots,N_{p_0}$.

Next, a classical result of Stirling \cite[p.151]{Ti} says
\begin{equation}
\log\boldsymbol{\Gamma}(z)=\Big(z-\frac{1}{2}\Big)\log z-z+\frac{1}{2}\log(2\pi)+\int_0^{+\infty}\frac{[t]-t+\frac{1}{2}}{t+z}dt,\nonumber
\end{equation}
which combined with a version of the Lebesgue's convergence theorem leads to
\begin{equation}
f(z)=\frac{\boldsymbol{\Gamma}'}{\boldsymbol{\Gamma}}(z)=\log z-\frac{1}{2z}-\int_0^{+\infty}\frac{[t]-t+\frac{1}{2}}{(t+z)^2}dt,\nonumber
\end{equation}
so that inductively for $n=1,2,\ldots$, one deduces
\begin{equation}
\begin{aligned}
f^{(n)}(z)&=(-1)^{n-1}\left\{\frac{(n-1)!}{z^n}+\frac{n!}{2z^{n+1}}+(n+1)!\int_0^{+\infty}\frac{[t]-t+\frac{1}{2}}{(t+z)^{n+2}}dt\right\}.\nonumber
\end{aligned}
\end{equation}
It is thus quite straightforward to verify
\begin{equation}
f(z)=\log z+o(1)=\log z(1+o(1))\nonumber
\end{equation}
and
\begin{equation}
f^{(n)}(z)=\frac{(-1)^{n-1}(n-1)!}{z^n}(1+o(1))\nonumber
\end{equation}
for every $n=1,2,\ldots$, and hence
\begin{equation}\label{Eq2.13}
\frac{f'}{f^2}(z)=\frac{1}{z(\log z)^2}(1+o(1))~\text{and}~\frac{f''}{ff'}(z)=-\frac{1}{z\log z}(1+o(1)),
\end{equation}
all uniformly on $\mathbf{D}:=\left\{z:-\frac{5\pi}{6}\leq\arg z\leq\frac{5\pi}{6}\right\}$, provided $\n{z}$ is sufficiently large.

Now, recall $c_n=\frac{n(n-1)}{2}$ and $\varepsilon_1=\varepsilon_2=0$.
One employs \eqref{Eq2.13} to see that
\begin{equation}
{\small\begin{aligned}
\varepsilon_3&=\frac{f''}{ff'}=-\frac{1}{z\log z}(1+o(1)),\\
\varepsilon_4&=\varepsilon_3+\frac{\varepsilon_3'}{f}+\Big(\frac{f'}{f^2}+\frac{f''}{ff'}\Big)(c_3+\varepsilon_3)=-\frac{4}{z\log z}(1+o(1)),\\
\varepsilon_5&=\varepsilon_4+\frac{\varepsilon_4'}{f}+\Big(2\frac{f'}{f^2}+\frac{f''}{ff'}\Big)(c_4+\varepsilon_4)=-\frac{10}{z\log z}(1+o(1)),\\
&\vdots\\
\varepsilon_n&=-\frac{n(n-1)(n-2)}{6}\frac{1}{z\log z}(1+o(1))~\text{upon assumption, and thus}\\
\varepsilon_{n+1}&=\varepsilon_n+\frac{\varepsilon_n'}{f}+\Big[(n-2)\frac{f'}{f^2}+\frac{f''}{ff'}\Big](c_n+\varepsilon_n)
=-\Big[\frac{n(n-1)(n-2)}{6}+c_n\Big]\frac{1}{z\log z}(1+o(1))\\
&=-\Big[\frac{n(n-1)(n-2)}{6}+\frac{n(n-1)}{2}\Big]\frac{1}{z\log z}(1+o(1))=-\frac{n(n^2-1)}{6}\frac{1}{z\log z}(1+o(1)),\nonumber
\end{aligned}}
\end{equation}
which further implies, uniformly on $\mathbf{D}$ for sufficiently large $\n{z}$, via \eqref{Eq2.6} and \eqref{Eq2.7},
\begin{equation}
G(z)=\frac{\ell(\ell-1)n^2}{2}\Big(1-\frac{\ell n+n-3}{3}\frac{1}{z\log z}\Big)+O\Big(\frac{1}{z(\log z)^2}\Big)=\frac{\ell(\ell-1)n^2}{2}(1+o(1)).\nonumber
\end{equation}

From now on, we shall focus entirely on $\left\{z_k\right\}^{+\infty}_{k=1}\varsubsetneq\mathbf{D}$ and observe that
\begin{equation}
F^q(z_k)b_{q,r}(\boldsymbol{\zeta},\boldsymbol{\zeta}',\ldots,\boldsymbol{\zeta}^{(m)})(z_k)H^r(z_k)\nonumber
\end{equation}
is equal to, in view of $H(z_k)=\frac{f'}{f^2}(z_k)G(z_k)=\frac{\ell(\ell-1)n^2}{2z_k(\log z_k)^2}(1+o(1))$,
\begin{equation}
\Big[\frac{\ell(\ell-1)n^2}{2}\Big]^rb_{q,r}(\boldsymbol{\zeta},\boldsymbol{\zeta}',\ldots,\boldsymbol{\zeta}^{(m)})(z_k)\frac{(\log z_k)^{nq-2r}}{(z_k)^r}(1+o(1))\nonumber
\end{equation}
when $k\to+\infty$, where the indices either satisfy $r=r_0$ with $0\leq q\leq q_0$ or satisfy $r_0<r\leq N_{p_0}$ with $0\leq q\leq M_{p_0}$.
As a result, the term
\begin{equation}
F^{q_0}(z_k)b_{q_0,r_0}(\boldsymbol{\zeta},\boldsymbol{\zeta}',\ldots,\boldsymbol{\zeta}^{(m)})(z_k)H^{r_0}(z_k),\nonumber
\end{equation}
among all the possible terms appearing in \eqref{Eq2.11}, dominates in growth when $k\to+\infty$.
In fact, for sufficiently large $k$, if $r=r_0$ with $0\leq q<q_0$, one derives
\begin{equation}
\frac{\n{\log z_k}^{nq}}{\n{z_k}^{r_0}\n{\log z_k}^{2r_0}}\ll\frac{\n{\log z_k}^{nq_0}}{\n{z_k}^{r_0}\n{\log z_k}^{2r_0}},\nonumber
\end{equation}
while if $r_0<r\leq N_{p_0}$ with $0\leq q\leq M_{p_0}$, one derives
\begin{equation}
\begin{aligned}
\n{\log z_k}^{nq}&\leq\n{\log z_k}^{nM_{p_0}+nq_0}\\
\n{z_k}^r\n{\log z_k}^{2r}&\gg\n{z_k}^{r_0}\n{\log z_k}^{nM_{p_0}+2r_0}.\nonumber
\end{aligned}
\end{equation}
So, since $\ell\geq2$ and $n\geq1$, one sees from \eqref{Eq2.11} and \eqref{Eq2.12}, as $k\to+\infty$, that
\begin{equation}\label{Eq2.14}
\n{P_{p_0}\Big(\boldsymbol{\zeta},\boldsymbol{\zeta}',\ldots,\boldsymbol{\zeta}^{(m)};1,\frac{\boldsymbol{\Gamma}^{(n)}}{\boldsymbol{\Gamma}},\frac{\boldsymbol{\Gamma}^{(\ell n)}} {\boldsymbol{\Gamma}}\Big)(z_k)}\geq\frac{1}{3}\frac{\n{\log z_k}^{nq_0-2r_0}}{\n{z_k}^{r_0}}.
\end{equation}

When $p_0=L$, then, by the definition of $p_0$ and \eqref{Eq2.14}, it yields that
\begin{equation}
\begin{aligned}
&P(\boldsymbol{\zeta},\boldsymbol{\zeta}',\ldots,\boldsymbol{\zeta}^{(m)};\boldsymbol{\Gamma},\boldsymbol{\Gamma}^{(n)},\boldsymbol{\Gamma}^{(\ell n)})(z_k)\\
=\,&P_L(\boldsymbol{\zeta},\boldsymbol{\zeta}',\ldots,\boldsymbol{\zeta}^{(m)};\boldsymbol{\Gamma},\boldsymbol{\Gamma}^{(n)},\boldsymbol{\Gamma}^{(\ell n)})(z_k)\\
=\,&\boldsymbol{\Gamma}^L(z_k)P_L\Big(\boldsymbol{\zeta},\boldsymbol{\zeta}',\ldots,\boldsymbol{\zeta}^{(m)}
;1,\frac{\boldsymbol{\Gamma}^{(n)}}{\boldsymbol{\Gamma}},\frac{\boldsymbol{\Gamma}^{(\ell n)}} {\boldsymbol{\Gamma}}\Big)(z_k)\neq0\nonumber
\end{aligned}
\end{equation}
for sufficiently large $k$, which contradicts our assumption \eqref{Eq1.2}.

When $p_0<L$, by virtue of another result of Stirling \cite[p.151]{Ti} which says
\begin{equation}
\n{\boldsymbol{\Gamma}\Big(\frac{3}{4}+\mathtt{i}y\Big)}=e^{-\frac{1}{2}\pi\n{y}}\n{y}^{\frac{1}{4}}\sqrt{2\pi}(1+o(1))\nonumber
\end{equation}
as $\n{y}\to+\infty$, along with $\frac{\n{\log z}^\imath}{\n{z}^\jmath\n{\log z}^{2\jmath}}\to0$ as $\n{z}\to+\infty$ if $\jmath>0$ for nonnegative integers $\imath,\jmath$, one easily observes, in view of \eqref{Eq2.1}, \eqref{Eq2.2}, \eqref{Eq2.3}, \eqref{Eq2.12} and \eqref{Eq2.14}, that
\begin{equation}\label{Eq2.15}
\begin{aligned}
&\n{\frac{P(\boldsymbol{\zeta},\boldsymbol{\zeta}',\ldots,\boldsymbol{\zeta}^{(m)};\boldsymbol{\Gamma},\boldsymbol{\Gamma}^{(n)},\boldsymbol{\Gamma}^{(\ell n)})(z_k)} {\boldsymbol{\Gamma}^L(z_k)}}\\
=&\n{\sum_{p=p_0}^L\frac{1}{\boldsymbol{\Gamma}^{L-p}(z_k)}P_p\Big(\boldsymbol{\zeta},\boldsymbol{\zeta}',\ldots,\boldsymbol{\zeta}^{(m)};
1,\frac{\boldsymbol{\Gamma}^{(n)}}{\boldsymbol{\Gamma}},\frac{\boldsymbol{\Gamma}^{(\ell n)}}{\boldsymbol{\Gamma}}\Big)(z_k)}\\
\geq&\,\frac{1}{6}\exp\Big(\frac{L-p_0}{2}\pi\n{y_k}\Big)(\n{y_k}^{\frac{1}{4}}\sqrt{2\pi})^{p_0-L}\frac{\n{\log z_k}^{nq_0-2r_0}}{\n{z_k}^{r_0}}\\
&-C_1\exp\Big(\frac{L-p_0-1}{2}\pi\n{y_k}\Big)(\n{y_k}^{\frac{1}{4}}\sqrt{2\pi})^{p_0+1-L}\n{\log z_k}^{C_2}\to+\infty
\end{aligned}
\end{equation}
as $k\to+\infty$ for some constants $C_1,C_2>0$ depending on $C_0,L,\ell,n$.
Hence,
\begin{equation}
P(\boldsymbol{\zeta},\boldsymbol{\zeta}',\ldots,\boldsymbol{\zeta}^{(m)};\boldsymbol{\Gamma},\boldsymbol{\Gamma}^{(n)},\boldsymbol{\Gamma}^{(\ell n)})(z_k)\neq0\nonumber
\end{equation}
for sufficiently large $k$, which again contradicts the hypothesis \eqref{Eq1.2}.
\end{proof}

\begin{claim}\label{clm2.2}
For each $0\leq p\leq L$ and all $z\in\mathbf{C}$, when
\begin{equation}\label{Eq2.16}
P_p(\boldsymbol{\zeta},\boldsymbol{\zeta}',\ldots,\boldsymbol{\zeta}^{(m)};\boldsymbol{\Gamma},\boldsymbol{\Gamma}^{(n)},\boldsymbol{\Gamma}^{(\ell n)})(z)\equiv0,
\end{equation}
then necessarily the polynomial $P_p(u_0,u_1,\ldots,u_m;v_0,v_1,v_2)$ vanishes identically.
\end{claim}

\begin{proof}
When $p=0$, by definition, $P_0(u_0,u_1,\ldots,u_m;v_0,v_1,v_2)$ is a polynomial in $u_0,u_1,\ldots,u_m$ alone; so, from the solutions \cite{MB,Os} to the question posted by Hilbert, $P_0(\boldsymbol{\zeta},\boldsymbol{\zeta}',\ldots,\boldsymbol{\zeta}^{(m)})(z)\equiv0$ leads to $P_0(u_0,u_1,\ldots,u_m)\equiv0$ immediately.

Henceforth, assume $p>0$.
For simplicity, we write $p=p_0$ and use the expression \eqref{Eq2.11} and all its associated notations.
Next, we prove that $b_{q,r}(u_0,u_1,\ldots,u_m)\equiv0$ for every $0\leq q\leq M_{p_0}$ and $0\leq r\leq N_{p_0}$.
To this end, we first show that each one of
\begin{equation}
b_{M_{p_0},0}(u_0,u_1,\ldots,u_m),b_{M_{p_0}-1,0}(u_0,u_1,\ldots,u_m),\ldots,b_{0,0}(u_0,u_1,\ldots,u_m)\nonumber
\end{equation}
must be identically equal to zero.
Let's start with $b_{M_{p_0},0}(u_0,u_1,\ldots,u_m)$, which is a polynomial in $u_0,u_1,\ldots,u_m$, and suppose it doesn't vanish identically.
Then, following what we have done in Claim \ref{clm2.1}, one has \eqref{Eq2.12} (or its analogue for this newly chosen $p_0$).
Among all the terms of $P_{p_0}\big(\boldsymbol{\zeta},\boldsymbol{\zeta}',\ldots,\boldsymbol{\zeta}^{(m)};1,\frac{\boldsymbol{\Gamma}^{(n)}}{\boldsymbol{\Gamma}},\frac{\boldsymbol{\Gamma}^{(\ell n)}} {\boldsymbol{\Gamma}}\big)(z_k)$ as described in \eqref{Eq2.11}, the term
\begin{equation}
F^{M_{p_0}}(z_k)b_{M_{p_0},0}(\boldsymbol{\zeta},\boldsymbol{\zeta}',\ldots,\boldsymbol{\zeta}^{(m)})(z_k)\sim(\log z_k)^{nM_{p_0}}\nonumber
\end{equation}
dominates in growth for large $k$, since $\frac{\n{\log z_k}^\imath}{\n{z_k}^\jmath\n{\log z_k}^{2\jmath}}\to0$ when $k\to+\infty$ if $\jmath>0$ for nonnegative integers $\imath,\jmath$.
Thus, analogous to \eqref{Eq2.14}, we deduce from \eqref{Eq2.11} and \eqref{Eq2.12} that
\begin{equation}
\begin{aligned}
&P_{p_0}(\boldsymbol{\zeta},\boldsymbol{\zeta}',\ldots,\boldsymbol{\zeta}^{(m)};\boldsymbol{\Gamma},\boldsymbol{\Gamma}^{(n)},\boldsymbol{\Gamma}^{(\ell n)})(z_k)\\
=\,&\boldsymbol{\Gamma}^{p_0}(z_k)P_{p_0}\Big(\boldsymbol{\zeta},\boldsymbol{\zeta}',\ldots,\boldsymbol{\zeta}^{(m)};
1,\frac{\boldsymbol{\Gamma}^{(n)}}{\boldsymbol{\Gamma}},\frac{\boldsymbol{\Gamma}^{(\ell n)}}{\boldsymbol{\Gamma}}\Big)(z_k)\neq0\nonumber
\end{aligned}
\end{equation}
for all sufficiently large $k$, which however contradicts the hypothesis \eqref{Eq2.16}.
As a result, we see $b_{M_{p_0},0}(u_0,u_1,\ldots,u_m)\equiv0$.
The next term in queue is $b_{M_{p_0}-1,0}(u_0,u_1,\ldots,u_m)$ with
\begin{equation}
F^{M_{p_0}-1}(z_k)b_{M_{p_0}-1,0}(\boldsymbol{\zeta},\boldsymbol{\zeta}',\ldots,\boldsymbol{\zeta}^{(m)})(z_k)\sim(\log z_k)^{nM_{p_0}-n},\nonumber
\end{equation}
so that one can derive $b_{M_{p_0}-1,0}(u_0,u_1,\ldots,u_m)\equiv0$ in exactly the same manner; repeating this process, $b_{q,0}(u_0,u_1,\ldots,u_m)\equiv0$ follows for every $q=0,1,\ldots,M_{p_0}$.
Next, after the elimination of $H$ in \eqref{Eq2.11}, one can perform the preceding procedure again for
\begin{equation}
b_{M_{p_0},1}(u_0,u_1,\ldots,u_m),b_{M_{p_0}-1,1}(u_0,u_1,\ldots,u_m),\ldots,b_{0,1}(u_0,u_1,\ldots,u_m)\nonumber
\end{equation}
and observe that $b_{q,1}(u_0,u_1,\ldots,u_m)\equiv0$ for every $q=0,1,\ldots,M_{p_0}$.
Continuing like this, one arrives at $b_{q,r}(u_0,u_1,\ldots,u_m)\equiv0$ for all $q=0,1,\ldots,M_{p_0}$ and $r=0,1,\ldots,N_{p_0}$.
Through \eqref{Eq2.9} and \eqref{Eq2.10}, we can finally conclude that $P_{p_0}(u_0,u_1,\ldots,u_m;v_0,v_1,v_2)\equiv0$.
\end{proof}

It follows from \eqref{Eq2.1} and \eqref{Eq2.2} that the proof of Theorem \ref{thm1.1} is a straightforward consequence of Claims \ref{clm2.1} and \ref{clm2.2}, so that \eqref{Eq1.2} indeed leads to $P(u_0,u_1,\ldots,u_m;v_0,v_1,v_2)\equiv0$.

\vskip 12pt
{\small{\bf Acknowledgement.} Both authors acknowledge the financial support of the {\sl 2019 Summer Faculty Research Fellowships} from College of Arts and Sciences, Texas A\&M University at San Antonio.
Both authors warmly thank Dr. Wei Chen and Dr. Quingyan Wang for some communications.}


\end{document}